\documentclass[11pt]{article}
\usepackage[a4paper]{anysize}\marginsize{2.0cm}{2.0cm}{1.0cm}{1.0cm}
\pdfpagewidth=\paperwidth \pdfpageheight=\paperheight
\usepackage{pgf,tikz,hyperref}
\usetikzlibrary{arrows}
\usepackage{amsfonts,amssymb,amsthm,amsmath,eucal}
\usepackage{graphicx}
\usepackage{ltablex}
\usepackage{caption, booktabs}

\theoremstyle{plain}
\newtheorem{thm}{Theorem}[section]
\newtheorem{theorem}[thm]{Theorem}

\newtheorem{lemma}[thm]{Lemma}
\newtheorem{proposition}[thm]{Proposition}
\newtheorem{corollary}[thm]{Corollary}

\theoremstyle{definition}
\newtheorem{definition}[thm]{Definition}
\newtheorem{remark}[thm]{Remark}

\newtheorem{question}[thm]{Question}
\newtheorem{problem}[thm]{Problem}

\newtheorem{thevarthm}[thm]{\varthmname}

\newenvironment{varthm*}[1]{\trivlist\item[]{\bf #1.}\it}{\endtrivlist}

\newcommand\be{\begin{eqnarray*}}
\newcommand\ee{\end{eqnarray*}}

\newcommand\C{\mathbb C}

\newcommand\calo{\mathcal O}

\newcommand\calc{\mathcal C}

\newcommand\R{\mathbb R}

\renewcommand\L{\mathbb L}
\newcommand\jeden{{\mathbb I}}
\renewcommand\P{\mathbb P}

\newcommand\newop[2]{\def#1{\mathop{\rm #2}\nolimits}}
\newop\edim{edim}
\newop\Zeroes{Zeroes}
\newop\Jac{Jac}
\newop\Ass{Ass}
\newop\SL{SL}
\newop\PGL{{\P}GL}
\newop\Km{Km}
\newop\reg{reg}

\newcommand\geproci{\emph{geproci}}

\newcommand\keywords[1]{{\renewcommand\thefootnote{}\footnotetext{\textit{Keywords:} #1.}}}
\newcommand\subclass[1]{{\renewcommand\thefootnote{}\footnotetext{\textit{Mathematics Subject Classification (2010):} #1.}}}

\def\endproof{\hspace*{\fill}\endproofsymbol\endtrivlist}

\def\endproofsymbol{\frame{\rule[0pt]{0pt}{6pt}\rule[0pt]{6pt}{0pt}}}

\begin{document}

\author{Piotr Pokora, Tomasz Szemberg, and Justyna Szpond}
\title{Unexpected properties of the Klein configuration\\ of $60$ points in $\P^3$}
\date{\today}
\maketitle
\thispagestyle{empty}

\begin{abstract}
   Felix Klein in course of his study of the regular icosahedron and its symmetries encountered a highly symmetric configuration of $60$ points in $\P^3$. This configuration has appeared in various guises, perhaps post notably as the configuration of points dual to the $60$ reflection planes in the group $G_{31}$ in the Shephard-Todd list.

   In the present note we show that the $60$ points exhibit interesting properties relevant from the point of view of two paths of research initiated recently. Firstly, they give rise to two completely different unexpected surfaces of degree $6$. Unexpected hypersurfaces have been introduced by Cook II, Harbourne, Migliore, Nagel in 2018. One of unexpected surfaces associated to the configuration of $60$ points is a cone with a single singularity of multiplicity $6$ and the other has three singular points of multiplicities $4,2$ and $2$. Secondly, Chiantini and Migliore observed in 2020 that there are non-trivial sets of points in $\P^3$ with the surprising property that their general projection to $\P^2$ is a complete intersection. They found a family of such sets, which they called grids. An appendix to their paper describes an exotic configuration of $24$ points in $\P^3$ which is not a grid but has the remarkable property that its general projection is a complete intersection. We show that the Klein configuration is also not a grid and it projects to a complete intersections. We identify also its proper subsets, which enjoy the same property.
\end{abstract}

\subclass{MSC 14C20 \and MSC 14N20 \and MSC 13A15}


\section{Introduction}
   Arrangements of hyperplanes defined by finite complex reflection groups and sets of points corresponding to the duals of the hyperplanes are a rich source of examples and a testing ground for various conjectures in commutative algebra and algebraic geometry. It is not surprising that they find their place in new paths of research developed within these two branches of mathematics.

   In the present note we study the configuration $Z_{60}$ of $60$ points in $\P^3$ determined by the reflection group $G_{31}$ in the Shephard Todd list \cite{SheTod54}. This configuration has been known for long, its origins go back to Felix Klein's doctoral thesis \cite{KleinPhD}. We recall basic properties, relevant for our study in Section \ref{sec: history}. In section \ref{sec: group} we discuss another group acting on the configuration, discovered only recently by Cheltsov and Shramov \cite{CheltsovShramov19}. In fact, their article was a departure point for this work.

   Our main results are presented in the two subsequent sections. The first results concern unexpected surfaces admitted by the set $Z_{60}$. The concept of \emph{unexpected hypersurfaces} has been introduced first for curves in the ground breaking article \cite{CHMN} by Cook II, Harbourne, Migliore and Nagel
   and generalized to arbitrary dimensions in the subsequent
   article \cite{HMNT} by Harbourne, Migliore, Nagel and Teitler. Loosely speaking, a finite set of points $Z$
   in a projective space admits an unexpected hypersurface of degree $d$ if a general fat point
   or a finite number of general fat points impose less conditions on forms
   of degree $d$ vanishing at $Z$ than naively expected, see Definition \ref{def:unexpected hypersurface}
   for precise statement. In Theorems \ref{thm: unexpected cone of degree 6} and \ref{thm: unexpected with 3 singularities} we show that $Z_{60}$ admits two different types of unexpected surfaces. One is a cone with a single singular point of multiplicity $6$ and the other is a surface of degree $6$ with three singular points of multiplicities $4, 2$ and $2$. It seems to be the first case where a fixed set of points admits unexpected hypersurfaces of two different kinds: one with a single general point in the spirit of foundational article \cite{CHMN} and the other with multiple general points in the spirit of \cite{Szp18multi}. This is a new and quite unexpected phenomenon.

   Theorem \ref{thm: geproci for Z60} goes in different direction. Chiantini and Migliore realized that there are sets of points spanning the whole $\P^3$ with the striking property that their general projections to a plane are complete intersections. We say that the set has the \geproci \, property, see Definition \ref{def geproci}. In their recent work \cite{ChiantiniMigliore19} they construct a series of examples of such sets, which they call grids, as they result as intersection points of two families of lines in $\P^3$ organized so that lines in both families are disjoint but each line from one family intersects all lines from the other family. They showed that for a small number of points, grids are the only sets with the \geproci \, property. On the other hand, in the appendix to their work there is an example of a set of $24$ points which has the \geproci \, property but is not a grid. We show that the set $Z_{60}$ behaves in the same way. It is not a grid but it has the \geproci \, property. More precisely its projection
   from a general point in $\P^3$ to $\P^2$ is a complete intersection of curves of degree
   $6$ and $10$. This provides in particular a positive answer to Question 7.1 in \cite{ChiantiniMigliore19}.

   We work over the field of complex numbers. We show in Proposition \ref{prop:notR} that the configuration studied here does not exist over the reals.
\section{Klein's arrangement $(60_{15})$ -- a historical outline}\label{sec: history}
We will follow a classical construction by F. Klein with a modern glimpse. Let ${\rm Gr}(2,4) = Q$ be the Grassmannian of lines in $\mathbb{P}^{3}$ embedded via the Pl\"ucker embedding as a smooth quadric hypersurface in $\mathbb{P}^{5}$. Let $L = H\cap Q$ with $H$ being a hyperplane in $\mathbb{P}^{5}$, then $L$ is called as a \emph{linear line complex}. This object was studied by F. Klein in his PhD thesis, see for instance \cite{KleinPhD}. In particular, using Klein's language, there are six fundamental linear complexes corresponding to the choice of the coordinate planes $H_i=\left\{x_{i}=0\right\}$ for homogeneous coordinated $x_{0},\ldots, x_{5}$ on $\mathbb{P}^{5}$. The intersections
$$Q \cap H_{i_{1}} \cap H_{i_{2}} \cap H_{i_{3}} \cap H_{i_{4}}$$
with $0 \leq i_{1} < i_{2} < i_{3} < i_{4} \leq 5$ consist of exactly two points each. These lines corresponding to the points form a configuration of
$$2 \cdot \binom{6}{4} = 30 \text{ lines in } \mathbb{P}^{3}.$$
We denote their union by $\mathbb{L}_{30}$. These lines intersect by $3$ in $60$ points. Klein showed that these points can be divided into $15$ subsets of $4$ points. Each subset determines vertices of a fundamental tetrahedron. The edges of these tetrahedra are contained in the $30$ lines. The faces and the vertices are all mutually distinct, this gives a $(60_{15})$ configuration of $60$ planes and $60$ vertices: through each vertex there pass exactly $15$ planes, and each plane contains exactly $15$ vertices. Moreover, through each of the $30$ edges $6$ planes pass, and each edge contains $6$ vertices. It is well-known by a result due to Shephard and Todd \cite{SheTod54} that this arrangement is defined by the unitary reflection group, denoted there by $G_{11,520}$. In the Orlik-Solomon notation the symbol $\mathcal{A}_{2}^{3}(60)$ is used for Klein's $(60_{15})$-arrangement. Following Hunt's notation for arrangements of planes in $\mathbb{P}^{3}$ \cite{Hunt86}, we denote by $t_{i}$ the number of points where exactly $i\geq 3$ of planes from the arrangement intersect, and by $t_{j}(1)$ the number of lines contained in exactly $j\geq 2$ planes in the arrangement. For $\mathcal{A}_{2}^{3}(60)$ we have:

$$t_{15} = 60, \quad t_{6} = 480, \quad t_{4} = 960,$$
$$t_{6}(1) = 30, \quad t_{3}(1) = 320, \quad t_{2}(1) = 360.$$
It is worth noticing that the set $Z_{60}$ and $\mathbb{L}_{30}$ build a $(60_{3}, 30_{6})$ configuration of points and lines.

To the arrangement $\mathcal{A}_{2}^{3}(60)$ we can associate an interesting arrangement of $10$ quadrics -- these are called in literature \textit{Klein's fundamental quadrics}. It is worth emphasizing that through each of the $30$ lines described above there are $4$ fundamental quadrics containing that line. We refer to \cite{CheltsovShramov19} for the table of incidences between the $30$ lines and $10$ Klein's fundamental quadrics.
\section{A finite Heisenberg group and a group of order $80$}
\label{sec: group}
   Let $H_{2,2}$ be the subgroup of $\SL_4(\C)$ generated by the following four matrices:
   $$
   S_1=\left(\begin{array}{cccc}
            0 & 0 & 1 & 0\\
            0 & 0 & 0 & 1\\
            1 & 0 & 0 & 0\\
            0 & 1 & 0 & 0
            \end{array}\right),\;
   S_2=\left(\begin{array}{cccc}
            0 & 1 & 0 & 0\\
            1 & 0 & 0 & 0\\
            0 & 0 & 0 & 1\\
            0 & 0 & 1 & 0
            \end{array}\right),$$
   $$
   T_1=\left(\begin{array}{cccc}
            1 & 0 & 0 & 0\\
            0 & 1 & 0 & 0\\
            0 & 0 & -1 & 0\\
            0 & 0 & 0 & -1
            \end{array}\right),\;
   T_2=\left(\begin{array}{cccc}
            1 & 0 & 0 & 0\\
            0 & -1 & 0 & 0\\
            0 & 0 & 1 & 0\\
            0 & 0 & 0 & -1
            \end{array}\right).$$
   This group has order $32$ (note that all pairs of generators commute except of $S_iT_i=-T_iS_i$ for $i=1,2$.) and the center and the commutator of $H_{2,2}$ are both equal to
   $\left\{\jeden, -\jeden\right\}$, where $\jeden$ denotes the identity matrix of size $4$. We call $H_{2,2}$ the (finite) \emph{Heisenberg group}.

   Consider the natural projection $\phi: \SL_{4}(\mathbb{C}) \rightarrow {\rm PGL}_{4}(\mathbb{C})$ and for every group $G \subset \SL_{4}(\mathbb{C})$ we define $\overline{G} = \phi(G)$. Denote by $G_{80}$ the subgroup in $\SL_{4}(\mathbb{C})$ generated by $H_{2,2}$ and the following matrix
   $$ T=\frac{1+i}{2}\left(\begin{array}{cccc}
           -i & 0 & 0 & i\\
            0 & 1 & 1 & 0 \\
            1 & 0 & 0 & 1 \\
            0 & -i & i & 0
            \end{array}\right),$$
    and by $\overline{G_{80}}$ the image of $G_{80}$ in ${\rm PGL}_{4}(\mathbb{C})$.

    With the notation fixed, we are now in the position to endow the configurations described in Section \ref{sec: history}
    with coordinates. Checking all properties of Klein's arrangement listed in Section \ref{sec: history} boils thus to elementary linear algebra. First, one can show that Klein's fundamental quadrics are invariant under the action of the group $\overline{G_{80}}$, and this group splits them into two orbits (note that the quadrics are $H_{2,2}$ invariant):
\begin{align*}
\mathcal{Q}_{1} &= x^2 + y^2 + z^2 + w^2,\\
\mathcal{Q}_{2} =T(\mathcal{Q}_{1}) &= xw + zy,\\
\mathcal{Q}_{3} =T^2(\mathcal{Q}_{1}) &= xz + yw,\\
\mathcal{Q}_{4} =T^3(\mathcal{Q}_{1}) &= x^2 + y^2 - z^2 - w^2,\\
\mathcal{Q}_{5} =T^4(\mathcal{Q}_{1})&= x^2 - y^2 - z^2 + w^2,\\
\mathcal{Q}_{6} &= x^2 - y^2 + z^2 - w^2, \\
\mathcal{Q}_{7} =T(\mathcal{Q}_{6}) &= xw - yz, \\
\mathcal{Q}_{8} =T^2(\mathcal{Q}_{6}) &= xy + zw,\\
\mathcal{Q}_{9} =T^3(\mathcal{Q}_{1}) &= xy -zw,\\
\mathcal{Q}_{10} =T^4(\mathcal{Q}_{1}) &= xz - yw.
\end{align*}
Using equations of quadrics and taking into account that each pair of them intersects in $4$ of $30$ Klein's lines, we identify the equations of lines
\[\begin{array}{ll}
 \ell_{1} = V(x,y), & \ell_{2} = V(z-w, x-y), \\

 \ell_{3} = V(z-w,x+y), &\ell_{4} = V(z +i\cdot w, x +i\cdot y), \\

 \ell_{5} = V(z+i \cdot w, x-i\cdot y), & \ell_{6} = V(z+w,x-y), \\

 \ell_{7} = V(z+w, x+y), & \ell_{8} = V(z-i \cdot w, x+ i\cdot y), \\

 \ell_{9} = V(z - i\cdot w, x -i\cdot y), & \ell_{10} = V(z,x), \\

 \ell_{11} = V(y-w,x-z), & \ell_{12} = V(y-w,x+z), \\

 \ell_{13} = V(y + i\cdot w, x+ i\cdot z), & \ell_{14} = V(y + i \cdot w, x - i\cdot z), \\

 \ell_{15} = V(y+w, x-z), & \ell_{16} = V(y+w, x+z), \\

 \ell_{17} = V(y-i \cdot w, x + i \cdot z), & \ell_{18} = V(y - i \cdot w, x - i \cdot z), \\

 \ell_{19} = V(w,x), & \ell_{20} = V(y-z, x-w), \\

 \ell_{21} = V(y-z, x+w), & \ell_{22} = V(y + i\cdot z, x + i \cdot w), \\

 \ell_{23} = V(y + i \cdot z, x-i \cdot w), & \ell_{24} = V (y+z, x-w),\\

 \ell_{25} = V(y+z,x+w), & \ell_{26} = V(y - i \cdot z, x+ i \cdot w), \\

 \ell_{27} = V(y-i \cdot z, x - i\cdot w), & \ell_{28} = V(z,y), \\

 \ell_{29} = V(w,y), & \ell_{30} = V(w,z).

\end{array}\]
Finally, taking intersection points of lines, we identify coordinates of points in the $Z_{60}$ set
\[\begin{array}{lll}
P_{1} = [0:0:1:1] & P_{2} = [0:0:1:i] & P_{3} = [0:0:1:-1] \\

P_{4} = [0:0:1:-i] & P_{5} = [0:1:0:1] & P_{6} = [0:1:0:i] \\

P_{7} = [0:1:0:-1] & P_{8} = [0:1:0:-i] & P_{9} = [0:1:1:0] \\

P_{10} = [0:1:i:0] & P_{11} = [0:1:-1:0] & P_{12} = [0:1:-i:0] \\

P_{13} = [1:0:0:1] & P_{14} = [1:0:0:i] & P_{15} = [1:0:0:-1] \\

P_{16} = [1:0:0:-i] & P_{17} = [1:0:1:0] & P_{18} = [1:0:i:0]\\

P_{19} = [1:0:-1:0] & P_{20} = [1:0:-i:0] & P_{21} = [1:1:0:0]\\

P_{22} = [1:i:0:0] & P_{23} = [1:-1:0:0] & P_{24} = [1:-i:0:0] \\

P_{25} = [1:0:0:0] & P_{26} = [0:1:0:0] & P_{27} = [0:0:1:0]\\

P_{28} =[0:0:0:1] & P_{29} = [1:1:1:1] & P_{30} = [1:1:1:-1]\\

P_{31} = [1:1:-1:1] & P_{32} = [1:1:-1:-1] & P_{33} = [1:-1:1:1]\\

P_{34} = [1:-1:1:-1] & P_{35} = [1:-1:-1:1] & P_{36} = [1:-1:-1:-1] \\

P_{37} = [1:1:i:i] & P_{38} = [1:1:i:-i] & P_{39} = [1:1:-i:i]\\

P_{40} = [1:1:-i:-i] & P_{41} = [1:-1:i:i] & P_{42} = [1:-1:i:-i]\\

P_{43} = [1:-1:-i:i] & P_{44} = [1:-1:-i:-i] & P_{45} = [1:i:1:i]\\

P_{46} = [1:i:1:-i] & P_{47} = [1:-i:1:i] & P_{48} = [1:-i:1:-i] \\

P_{49} = [1:i:-1:i] & P_{50} = [1:i:-1:-i] & P_{51} = [1:-i:-1:i]\\

P_{52} = [1:-i:-1:-i] & P_{53} = [1:i:i:1] & P_{54} = [1:i:-i:1]\\

P_{55} = [1:-i:i:1] & P_{56} = [1:-i:-i:1] & P_{57} = [1:i:i:-1]\\

P_{58} = [1:i:-i:-1] & P_{59} = [1:-i:i:-1] & P_{60} = [1:-i:-i:-1].
\end{array}\]
The following lemma proved in \cite[Lemma 3.18]{CheltsovShramov19} gives useful geometric information about $Z_{60}$ and $\L_{30}$.
\begin{lemma}[Line-point incidences]
\label{lemma:Heseinebrg-lines-points}
The set $Z_{60} \subset \mathbb{P}^{3}$ contains all the intersection points of lines in $\L_{30}$. Moreover, for every point $P\in Z_{60}$ there are exactly three lines from $\L_{30}$ passing through $P$.
\end{lemma}
Now we turn to algebraic properties of the ideal of points in $Z_{60}$.
\begin{lemma}[Generators of $I(Z_{60})$]\label{lem: 24 generators of J}
The ideal $J=I(Z_{60})$ of points in $Z_{60}$ is generated by $24$ forms of degree $6$.
\end{lemma}
\begin{proof}
We want to show that the following, particularly nice forms, generate $J$, namely
\begin{equation}\label{eq: 24 generators}
\begin{array}{llll}
xy(x^4-y^4) & xz(z^4-x^4) & xw(x^4-w^4) & yz(y^4-z^4)\\
yw(w^4-y^4) & zw(z^4-w^4) & xy(z^4-w^4) & xz(y^4-w^4)\\
xw(y^4-z^4) & yz(x^4-w^4) & yw(x^4-z^4) & zw(x^4-y^4)\\
yw(x^2y^2-z^2w^2) & xw(x^2y^2-z^2w^2) & yz(x^2y^2-z^2w^2) & xz(x^2y^2-z^2w^2)\\
zw(x^2z^2-y^2w^2) & xw(x^2z^2-y^2w^2) & yz(x^2z^2-y^2w^2) & xy(x^2z^2-y^2w^2) \\
zw(y^2z^2-x^2w^2) & yw(y^2z^2-x^2w^2) & xz(y^2z^2-x^2w^2) & xy(y^2z^2-x^2w^2).
\end{array}
\end{equation}
To this end, we study first the diminished set $W=Z_{60}\setminus\{P_{25},P_{26},P_{27},P_{28}\}$,
i.e., the set $Z_{60}$ without the $4$ coordinate points in $\P^3$.

We claim that there is no form of degree $5$ vanishing along $W$. Note that the coordinate points in $\P^3$ lie in pairs on lines $\ell_1$, $\ell_{10}$, $\ell_{19}$, $\ell_{28}$, $\ell_{29}$ and $\ell_{30}$  from the set $\mathbb{L}_{30}$. In particular, the remaining $24$ lines $\mathbb{L}_{24}$ contain each still $6$ points from the set $W$. Note that the coordinate points in $\P^3$ are not contained in quadrics $Q_1$, $Q_4$, $Q_5$, $Q_6$. Hence each quadric contains $12$ lines from $\mathbb{L}_{30}$. Suppose now that there is a surface $\Omega$ of degree $5$ vanishing along $W$. Then, by B\'ezout Theorem, $\Omega$ contains all lines in $\mathbb{L}_{24}$. Taking the intersection of $\Omega$ with irreducible quadric $Q_1$, we identify $12$ lines contained in $Q_1$ as lying in $\Omega$. Again, by B\'ezout Theorem, this is possible only if $Q_1$ is a component of $\Omega$. By the same token, quadrics $Q_4$, $Q_5$ and $Q_6$ are also components of $\Omega$. As this is clearly not possible, we conclude that $\Omega$ does not exist.

Hence the points in $W$ impose independent conditions on forms of degree $5$ in $\P^3$, i.e., we have
\[
H^1(\P^3;\mathcal{O}_{\P^3}(5)\otimes I(W))=0.
\]
By \cite[Theorem 1.8.3]{PAG}, this gives
\[
{\rm reg} (I(W))=6.
\]
It follows that $W$ imposes independent conditions on forms of degree $6$ as well, hence
\[
h^0(\P^3,\mathcal{O}_{\P^3}(6)\otimes I(W))=\binom{9}{3}-56=28.
\]
In addition to generators listed in \eqref{eq: 24 generators} we have the following $4$ generators:
\begin{align*}
g_1 & =2x^2y^2z^2-x^4w^2-y^4w^2-z^4w^2+w^6,\\
g_2 & =2x^2y^2w^2-x^4z^2-y^4z^2-w^4z^2+z^6,\\
g_3 & =2x^2z^2w^2-x^4y^2-z^4y^2-w^4y^2+y^6,\\
g_4 & =2y^2z^2w^2-y^4x^2-z^4x^2-w^4x^2+x^6.
\end{align*}
Now, it is easy to see that requiring vanishing at the $4$ coordinate points kills the above additional generators.
\end{proof}

\section{Unexpected hypersurfaces associated with $Z_{60}$}
   In the ground-breaking work \cite{CHMN} by Cook II, Harbourne, Migliore and Nagel introduced the concept of unexpected curves. This notion was generalized to arbitrary hypersurfaces in the subsequent article \cite{HMNT} by Harbourne, Migliore, Nagel and Teitler.
\begin{definition}\label{def:unexpected hypersurface}
   We say that a reduced set of points $Z\subset\P^N$ \emph{admits an unexpected hypersurface} of degree $d$
   if there exists a sequence of non-negative integers $m_1,\ldots,m_s$ such that for general points $P_1,\ldots,P_s$
   the zero-dimensional subscheme $P = m_1P_1+\ldots +m_sP_s$ fails to impose independent conditions on forms
   of degree $d$ vanishing along $Z$ and the set of such forms is non-empty. In other words, we have
   $$h^0(\P^N;\calo_{\P^N}(d)\otimes I(Z)\otimes I(P))>
     \max\left\{0, h^0(\P^N;\calo_{\P^N}(d)\otimes I(Z))-\sum_{i=1}^s\binom{N+m_s-1}{N}\right\}.$$
\end{definition}
   Following \cite[Definition 2.5]{ChiantiniMigliore19} we introduce also the following notion.
\begin{definition}[Unexpected cone property]
   Let $Z$ be a finite set of points in $\P^N$ and let $d$ be a positive integer. We say that $Z$ has the \emph{unexpected cone property} $\calc(d)$, if for a general point $P\in\P^3$, there exists an unexpected (in the sense of Definition \ref{def:unexpected hypersurface}) hypersurface $S_P$ of degree $d$ and multiplicity $d$ at $P$ passing through all points in $Z$.
\end{definition}

\begin{theorem}[Unexpected cone property of $Z_{60}$]\label{thm: unexpected cone of degree 6}
   The set $Z_{60}$ has the $\calc(6)$ property. Moreover, the unexpected cone of degree $6$ is unique.
\end{theorem}
\begin{proof}
   Let $P=(a:b:c:d)$ be a general point in $\P^3$. Then
\begin{align*}
F = & xy(x^4-y^4)cd(c^4-d^4)+xz(z^4-x^4)bd(b^4-d^4)+xw(x^4-w^4)bc(b^4-c^4)\\
+& yz(y^4-z^4)ad(a^4-d^4)+yw(w^4-y^4)ac(a^4-c^4)+zw(z^4-w^4)ab(a^4-b^4)\\
+& 5xy(z^4-w^4)cd(a^4-b^4)+5xz(y^4-w^4)bd(c^4-a^4)+5xw(y^4-z^4)bc(a^4-d^4)\\
+& 5yz(x^4-w^4)ad(b^4-c^4)+5yw(x^4-z^4)ac(d^4-b^4)+5zw(x^4-y^4)ab(c^4-d^4)\\
+& 10yw(x^2y^2-z^2w^2)ac(c^2d^2-a^2b^2)+10xw(x^2y^2-z^2w^2)bc(a^2b^2-c^2d^2)\\
+& 10yz(x^2y^2-z^2w^2)ad(a^2b^2-c^2d^2)+10xz(x^2y^2-z^2w^2)bd(c^2d^2-a^2b^2)\\
+& 10zw(x^2z^2-y^2w^2)ab(a^2c^2-b^2d^2)+10xw(x^2z^2-y^2w^2)bc(b^2d^2-a^2c^2)\\
+& 10yz(x^2z^2-y^2w^2)ad(b^2d^2-a^2c^2)+10xy(x^2z^2-y^2w^2)cd(a^2c^2-b^2d^2)\\
+& 10zw(y^2z^2-x^2w^2)ab(a^2d^2-b^2c^2)+10yw(y^2z^2-x^2w^2)ac(b^2c^2-a^2d^2)\\
+& 10xz(y^2z^2-x^2w^2)bd(b^2c^2-a^2d^2)+10xy(y^2z^2-x^2w^2)cd(a^2d^2-b^2c^2)
\end{align*}
   defines a cone of degree $6$ with the vertex at $P$. Being unexpected cone for $\mathcal{C}(6)$ follows immediately from Lemma \ref{lem: 24 generators of J} since a point of multiplicity $6$ is expected to impose $56$ conditions.

   The equation of $F$ has been found by Singular \cite{Singular} and can be verified by the script \cite{60script} accompanying our manuscript. Once the equation is there, the claimed properties can be checked, at least in principle, by hand. However, the highly symmetric form of $F$, with respect to the sets of variables $\left\{x,y,z,w\right\}$ and $\left\{a,b,c,d\right\}$ is not a coincidence. It was established in \cite{HMNT} that the BMSS-duality, observed first in \cite{BMSS}, implies that the equation of $F$, considered as a polynomial in variables $\left\{a,b,c,d\right\}$, describes the tangent cone at $P$ of the surface defined by $F$ in variables $\left\{x,y,z,w\right\}$. Since the set of zeroes of $F$ is a cone with vertex $P$, it is the same cone in both sets of variables.

   The property that $F$ is unique follows easily from the fact that for $P$ general the polynomial $F$ is irreducible (it is a cone over a smooth curve of degree $6$) and any other cone of degree $6$ with vertex at $P$ and multiplicity $6$ would intersect $F$ in $36$ lines. Taking $P$ general, away of the secant variety of $Z_{60}$, i.e., away of the union of lines through pairs of points in $Z_{60}$ these $36$ lines would not be enough to cover the whole set $Z_{60}$.
\end{proof}

\begin{theorem}[Unexpected surface with $3$ general points]
\label{thm: unexpected with 3 singularities}
   Let $P,Q_1,Q_2$ be general points in $\P^3$. Then there exists a unique surface of degree $6$ vanishing in all points of $Z_{60}$ with a point of multiplicity $4$ at $P$ and multiplicity $2$ at both $Q_1$ and $Q_2$.
\end{theorem}
\begin{proof}
   Let $P=(a:b:c:d)$. We consider first sextics vanishing at all the points in $Z_{60}$ and at a general point $P$ to order $4$. We are not able to write them explicitly down because the equations are too complex. In fact, the coefficients in front of monomials of degree $6$ in variables $x,y,z,w$ are polynomials of degree $45$ in variables $a,b,c,d$. This is quite surprising when confronted with the proof of Theorem \ref{thm: unexpected cone of degree 6}.

   Our approach is quite standard. We outline it here and refer to our script \cite{60script} for details. We  build an interpolation matrix whose columns are the $24$ generators of $J=I(Z_{60})$. In the rows we write down one by one all the $20$ differentials of order $4$ of the generators and evaluate them at $P$. This gives a $20\times 24$ matrix. Even though many of its coefficients are $0$, the matrix is still too large to reproduce here. Nevertheless it is simple enough, so that Singular can compute its rank, which is $15$. That means that vanishing at $P$ to order $4$ imposes only $15$ instead of the expected $20$ conditions on generators of $J$. With this fact established, the remaining part of the proof is easy. We have linear system of sextics of dimension $24-15=9$, so it allows $2$ singularities $Q_1$ and $Q_2$ anywhere. Interestingly, no symbolic algebra program we asked, was able to determine the coefficients of the unique sextic vanishing at $P$ to order $4$ and at $Q_1,Q_2$ to order $2$. We expect that their coefficients in coordinates of the singular points are huge.
\end{proof}
\begin{remark}
   Note that unexpected hypersurfaces with multiple singular points seem to be quite rare. The sextics described in Theorem \ref{thm: unexpected with 3 singularities} are the only example, we are aware of, apart of a series of examples related to Fermat-type arrangements constructed by the third author in \cite{Szp18multi}.
\end{remark}

\section{Projections of $Z_{60}$}
   In this section we study general projections of $Z_{60}$. Our motivation comes from a recent work of Chiantini and Migliore \cite{ChiantiniMigliore19}.
\begin{definition}[\geproci \, property]
\label{def geproci}
   We say that a finite set $Z\subset\P^3$ has a \emph{general projection complete intersection} property (\geproci \, in short), if its projection from a general point in $\P^3$ to $\P^2$ is a complete intersection.
\end{definition}
   Obvious examples of \geproci\, sets are complete intersections in $\P^3$ with the property that one of the intersecting surfaces is a plane. Thus it is interesting to study non-degenerate (i.e. not contained in a hyperplane) sets $Z$ with the \geproci\, property.

Chiantini and Migliore  observed in \cite{ChiantiniMigliore19} that such sets exist.
They distinguished grids as a wide class of sets which enjoy the \geproci\,  property.
\begin{definition}[$(a,b)$-grid]
Let $a, b$ be a positive integers. A set $Z$ of $ab$ points in $\mathbb{P}^{3}$ is an $(a,b)$ - grid if there exists a set of pairwise skew lines $\ell_{1}, ...,\ell_{a}$ and a set of $b$ skew lines $\ell_{1}', ..., \ell_{b}'$, such that
$$Z = \{\ell_{i} \cap \ell_{j}' \, :\, i \in \{1, ..., a\}, j \in \{1, ..., b\}\}.$$
This means in particular that for all $i,j$ the lines $\ell_{i}, \ell_j'$ are different and incident.
\end{definition}
Indeed, projecting to $\P^2$ lines in $\P^3$ from a general point in $\P^3$ results again in lines in $\P^2$. Let $\pi: \P^3\dashrightarrow \P^2$ be such a projection. Then $\pi(Z)$ is the complete intersection of curves $C=\pi(\ell_1)+\ldots +\pi(\ell_a)$ and $D=\pi(\ell_1')+\ldots +\pi(\ell_b')$.

Appendix to \cite{ChiantiniMigliore19} contains examples of sets of points in $\P^3$, which are not $(a,b)$-grids, but which have the \geproci\, property. Such sets seem to be extremely rare, which motivated the following problem.
\begin{question}[Chiantini, Migliore, \cite{ChiantiniMigliore19} Question 7.1]
Are there any examples of $ab$ points in $\mathbb{P}^{3}$, other than $ab = 12, 16, 20$ or $24$, that are not $(a, b)$-grids but that have a general projection that is a complete intersection in $\mathbb{P}^{2}$ of type $(a, b)$?
\end{question}
   It turns out that the set $Z_{60}$ is not a grid and yet it has the \geproci \, property.
\begin{proposition}
The set $Z_{60}$ is not an $(a,b)$-grid for any choice of $a,b$.
\end{proposition}
\begin{proof}
   The only possibilities for the numbers $a$ and $b$ up to their order are:
   $$(2,30),\;\; (3,20),\;\; (4,15),\;\; (5,12)\; \mbox{ and }\;(6,10).$$
   It is easy to check that each of distinguished lines
   $\ell_1,\ldots,\ell_{30}$ contains $6$ points from $Z_{60}$ and this is the highest number of collinear points in $Z_{60}$. There are additional $320$ lines meeting $Z_{60}$ in $3$ points. In any case, the number of collinear points in $Z_{60}$ is too small to allow the grid structure.
\end{proof}
\begin{theorem}[$Z_{60}$ is \geproci]\label{thm: geproci for Z60}
   The set $\pi(Z_{60})$ has the \geproci \, property. More precisely, its general projection to $\P^2$ is a complete intersection of curves of degree $6$ and $10$.
\end{theorem}
\begin{proof}
Let $P=(a:b:c:d)$ be a general point in $\P^3$ and let $\pi$ be the rational map
   $$\P^3\ni(x:y:z:w)\dashrightarrow (ay-bx:bz-cy:cw-dz)\in\P^2.$$
   Then, the cone $F$ associated to $P$ in the proof of Theorem \ref{thm: unexpected cone of degree 6} projects to the following curve of degree $6$ in variables $(s:t:u)$
\begin{align*}
C_6 = & b(a^4-b^4)tu(t^4-u^4)+c(a^4-c^4)su(u^4-s^4)+d(a^4-d^4)st(s^4-t^4)\\
+ & 5b(d^4-c^4)s^4tu+5c(b^4-d^4)st^4u+5d(c^4-b^4)stu^4\\
+ & 10b(a^2d^2-b^2c^2)s^2t^3u+10c(a^2d^2-b^2c^2)s^3t^2u+10d(a^2c^2-b^2d^2)s^3tu^2\\
+ & 10b(b^2d^2-a^2c^2)s^2tu^3+10c(a^2b^2-c^2d^2)st^2u^3+10d(c^2d^2-a^2b^2)st^3u^2.
\end{align*}
   By construction, the projection of $Z_{60}$ is contained in $C_6$. Somewhat surprisingly, there is a certain ambiguity in the choice of a curve of degree $10$ cutting out on $C_6$ precisely the set $Z_{60}$. The most appealing way comes from the geometry of the arrangement of lines $\mathbb{L}_{30}$.
   Using explicit equations of lines $\ell_i$ and coordinates of points $P_i$, it is easy to check that there are $6$ ways of choosing $10$ disjoint lines among $\left\{\ell_1,\ldots,\ell_{30}\right\}$ covering the set $Z_{60}$.
   These selections are indicated in Table \ref{tab: 10 disjoint lines}.
\begin{tabularx}{\linewidth}{|l|c|c|c|c|c|c|}
   \caption{The division of $60$ lines in $6$ groups of $10$ disjoint lines.}
   \label{tab: 10 disjoint lines}
\\
\toprule
   \hline
    & A & B & C & D & E & F\\
   \hline
   $\ell_{1}$ & + & + & & & & \\
   \hline
   $\ell_2$ & & & & & + & + \\
   \hline
   $\ell_3$ & & & + & + & & \\
   \hline
   $\ell_4$ & & & + & & + & \\
   \hline
   $\ell_5$ & & & & + & & + \\
   \hline
   $\ell_6$ & & & + & + & & \\
   \hline
   $\ell_7$ & & & & & + & + \\
   \hline
   $\ell_8$ & & & & + & & + \\
   \hline
   $\ell_9$ & & & + & & + & \\
   \hline
   $\ell_{10}$ & & & + & & & + \\
   \hline
   $\ell_{11}$ & & + & & + & & \\
   \hline
   $\ell_{12}$ & + & & & & + & \\
   \hline
   $\ell_{13}$ & + & & & + & & \\
   \hline
   $\ell_{14}$ & & + & & & + & \\
   \hline
   $\ell_{15}$ & + & & & & + & \\
   \hline
   $\ell_{16}$ & & + & & + & & \\
   \hline
   $\ell_{17}$ & & + & & & + & \\
   \hline
   $\ell_{18}$ & + & & & + & & \\
   \hline
   $\ell_{19}$ & & & & + & + & \\
   \hline
   $\ell_{20}$ & + & & + & & & \\
   \hline
   $\ell_{21}$ & & + & & & & + \\
   \hline
   $\ell_{22}$ & & + & + & & & \\
   \hline
   $\ell_{23}$  & + & & & & & + \\
   \hline
   $\ell_{24}$ & & + & & & & + \\
   \hline
   $\ell_{25}$ & + & & + & & & \\
   \hline
   $\ell_{26}$ & + & & & & & + \\
   \hline
   $\ell_{27}$ & & + & + & & & \\
   \hline
   $\ell_{28}$ & & & & + & + & \\
   \hline
   $\ell_{29}$ & & & + & & & + \\
   \hline
   $\ell_{30}$ & + & + & & & & \\
   \hline
\end{tabularx}
   Since the curve $C_6$ is irreducible, the image under the projection of any selection of $10$ disjoint lines out of $\L_{30}$, cuts $C_6$ in exactly $60$ distinct points. It follows that $C_6$ and lines intersect transversally and thus the intersection is scheme theoretic.

   We complete our considerations providing explicit equations of images $\ell_i'$ of the $30$ lines from~$\mathbb{L}_{30}$.
\begin{equation}\label{eq:30 lines}
\begin{array}{l}
 \ell_{1}' = s, \\
 \ell_{2}' = (c^2-cd)s+(ac-ad-bc+bd)t+(-ab+b^2)u, \\
 \ell_{3}' = (c^2-cd)s+(ac-ad+bc-bd)t+(-ab-b^2)u, \\
 \ell_{4}' = (i\cdot cd+c^2)s+(i\cdot(ad+bc)+ac-bd)t+(i\cdot ab-b^2)u, \\
 \ell_{5}' = (i\cdot cd+c^2)s+(i\cdot (ad-bc)+ac+bd)t+(i\cdot ab+b^2)u \\
 \ell_{6}' = (c^2+cd)s+(ac+ad-bc-bd)t+(ab-b^2)u, \\
 \ell_{7}' = (c^2+cd)s+(ac+ad+bc+bd)t+(ab+b^2)u, \\
 \ell_{8}' = (-i\cdot cd+c^2)s+(-i\cdot(ad-bc)+ac+bd)t+(-i\cdot ab+b^2)u, \\
 \ell_{9}' = (-i\cdot cd+c^2)s+(-i\cdot(ad+bc)+ac-bd)t+(-i\cdot ab-b^2)u, \\
 \ell_{10}' = cs+at, \\
 \ell_{11}' = (bc-cd)s+(-ad+bc)t+(-ab+bc)u, \\
 \ell_{12}' = (bc-cd)s+(-ad-bc)t+(-ab-bc)u, \\
 \ell_{13}' = (i\cdot cd+bc)s+i\cdot(ad-bc)t+(i\cdot ab-bc)u, \\
 \ell_{14}' = (i\cdot cd+bc)s+i\cdot(ad+bc)t+(i\cdot ab+bc)u, \\
 \ell_{15}' = (bc+cd)s+(ad+bc)t+(ab-bc)u, \\
 \ell_{16}' = (bc+cd)s+(ad-bc)t+(ab+bc)u, \\
 \ell_{17}' = (-i\cdot cd+bc)s-i\cdot(ad+bc)t+(-i\cdot ab+bc)u, \\
 \ell_{18}' = (-i\cdot cd+bc)s-i\cdot(ad-bc)t+(-i\cdot ab-bc)u, \\
 \ell_{19}' = cds+adt+abu, \\
 \ell_{20}' = (bc-c^2)s+(-ac+bd)t+(b^2-bc)u, \\
 \ell_{21}' = (bc-c^2)s+(-ac-bd)t+(-b^2+bc)u, \\
 \ell_{22}' = (i\cdot c^2+bc)s+i\cdot(ac-bd)t+(-i\cdot b^2+bc)u, \\
 \ell_{23}' = (i\cdot c^2+bc)s+i\cdot(ac+bd)t+(i\cdot b^2-bc)u, \\
 \ell_{24}' = (bc+c^2)s+(ac+bd)t+(b^2+bc)u,\\
 \ell_{25}' = (bc+c^2)s+(ac-bd)t+(-b^2-bc)u, \\
 \ell_{26}' = (-i\cdot c^2+bc)s-i\cdot(ac+bd)t+(-i\cdot b^2-bc)u, \\
 \ell_{27}' = (-i\cdot c^2+bc)s-i\cdot (ac-bd)t+(i\cdot b^2+bc)u, \\
 \ell_{28}' = t, \\
 \ell_{29}' = dt+bu, \\
 \ell_{30}' = u.
\end{array}
\end{equation}
\end{proof}
\begin{remark}
   One might expect that the projection of $10$ disjoint lines in $\L_{30}$ is somehow special. However, it can be checked that the contrary situation holds since the lines form the star configuration -- they intersect only in pairs producing $45$ double intersection points. The intersections take place away from the $C_6$ curve.
\end{remark}
   We derive for completeness the following corollary to Theorem \ref{thm: geproci for Z60}.
\begin{corollary}[Subsets of $Z_{60}$ with the \geproci \, property]\label{cor: Z60 minus lines}
   Removing $6$ collinear points from $Z_{60}$ produces a set $Z$ of $54$ points with the \geproci \, property. Their projection is a complete intersection of the curve of degree $6$ and now the remaining $9$ lines covering $Z$.
   This procedure can be repeated with remaining sets of $6$ collinear points. Thus we get sets of $60, 54, 48, 42, 36, 30$, and $24$  points, which are not $(a,b)$-grids, with the \geproci \, property.
\end{corollary}
\begin{proof}
   The only feature to check is if the obtained sets of points are not grids. Removing points in $Z_{60}$ from one, two, or three lines from the same set of lines fixed among $A,B,C,D,E$ or $F$ certainly does not lead to any grid, because the maximal number of collinear points in $Z_{60}$ is $6$. Removing the fourth line, we get a set $V_4$ of $36$ points, which could, in principle, be a $(6,6)$-grid. However, a direct check (supported by computer, but also manageable by hand) shows that the only lines with $6$ points come from the set covering $Z_{60}$ selected for the procedure, say $A$. Going down to $30$ points, we obtain a set $V_5$ and we detect only two lines with $5$ points from $V_5$ (these lines come from two different families, for example if we run the procedure with the family $A$, then they belong to $D$ and $F$). The same lines are the only two lines with $4$ points from $V_6$ when we pass from $V_5$ to $V_6$ removing another set of $6$ collinear points. However, further removal fails as noted in Remark \ref{rem:18_is_grid}.
\end{proof}
\begin{remark}\label{rem:18_is_grid}
   Taking out collinear points as in Corollary \ref{cor: Z60 minus lines} we can arrive also at a set of $18$ points. Somewhat unexpectedly such sets turn out to grids.
\end{remark}
On the other hand, it is interesting to note that whereas the curve  $C_6$ of degree $6$ vanishing along $\pi(Z_{60})$ is unique, there are six ways to choose a completely reducible (i.e. splitting in lines) degree $10$ curve cutting out $\pi(Z_{60})$ on $C_6$. The example studied in this work motivates the following definition.
\begin{definition}[Half grid]
Let $a$, $d$ be positive integers. A set $Z$ of $ad$ points in $\P^3$ is an $(a,d)$ - half grid if there exists a set of mutually skew lines $\ell_1,\ldots, \ell_a$ covering $Z$ and a general projection of $Z$ to a plane is a complete intersection of images of $a$ lines with a (possibly reducible ) curve of degree $d$.
\end{definition}
It is clear that $Z_{60}$ is a half grid. Moreover, any grid is a half grid. It is also clear that the point in $Z$ are equidistributed over the lines. Taking $6$ collinear points out of $Z_{60}$ results also in a half grid.
\begin{problem}
Are there any half grids, but not grids in $\P^3$ other than $Z_{60}$ and its subgrids?
\end{problem}
\section{More subsets of $Z_{60}$ with the \geproci \,property}
In the Appendix to \cite{ChiantiniMigliore19} the first sets of points in $\mathbb{P}^{3}$ with the \geproci \, property have been identified. They are subsets of the set of points, which we call here $Z_{24}$, associated with the $F_{4}$ root system. This set is a subset of $Z_{60}$. We keep the numbering of points:
\[\begin{array}{lll}
P_{1} = [0:0:1:1] &  P_{3} = [0:0:1:-1] & P_{5} = [0:1:0:1] \\

P_{7} = [0:1:0:-1] &  P_{9} = [0:1:1:0] & P_{11} = [0:1:-1:0] \\

P_{13} = [1:0:0:1] & P_{15} = [1:0:0:-1] & P_{17} = [1:0:1:0] \\

P_{19} = [1:0:-1:0] & P_{21} = [1:1:0:0] & P_{23} = [1:-1:0:0] \\

P_{25} = [1:0:0:0] & P_{26} = [0:1:0:0] & P_{27} = [0:0:1:0] \\

P_{28} = [0:0:0:1] & P_{29} = [1:1:1:1] & P_{30} = [1:1:1:-1] \\

P_{31} = [1:1:-1:1] & P_{32} = [1:1:-1:-1] & P_{33} = [1:-1:1:1] \\

P_{34} = [1:-1:1:-1] & P_{35} = [1:-1:-1:1] & P_{36} = [1:-1:-1:-1].
\end{array}\]
As it was shown in Appendix to \cite{ChiantiniMigliore19}, the set $Z_{24}$ is not any $(a,b)$-grid, but its general projection onto $\mathbb{P}^{2}$ is a complete intersection of curves of degree $4$ and $6$. Using our results from the previous sections we would like to describe the geometry of degree $6$ curves that are crucial to construct the complete intersection of $24$ points in the plane.  First of all, by an easy inspection, we can find a subset $\mathbb{L}_{18}$ of $\mathbb{L}_{30}$ consisting of $18$ lines such that these lines and the set $Z_{24}$ form a $(24_{3},18_{4})$ point-line configuration in $\mathbb{P}^{3}$. The indices of lines in $\L_{18}$ corresponding to those in \eqref{eq:30 lines} together with groups from Table \ref{tab: 10 disjoint lines} are presented below: 
\begin{align}\label{eq: 18 lines}
\begin{split}
\L_{18} = &\left\{ \ell_{1} (AB), \ell_{2} (EF), \ell_{3} (CD), \ell_{6} (CD), \ell_{7} (EF), \ell_{10} (CF),\ell_{11} (BD), \ell_{12} (AE), \ell_{15} (AE),\right.\\
  & \left.\ell_{16} (BD), \ell_{19} (DE), \ell_{20} (AC), \ell_{21} (BF), \ell_{24} (BF), \ell_{25} (AC), \ell_{28} (DE), \ell_{29} (CF), \ell_{30} (AB)
\right\}.
\end{split}
\end{align}
Choosing lines indexed by the same letter gives a subset of $6$ disjoint lines in $\L_{18}$ covering the set $Z_{24}$.
For example, the lines corresponding to $A$ are
$$\{\ell_1, \ell_{12}, \ell_{15}, \ell_{20}, \ell_{25}, \ell_{30}\}.$$
That the lines are disjoint is clear from the definition of groups $A,B,\ldots,F$. The covering property follows by a direct inspection. This shows, in particular, that $Z_{24}$ is a half grid.

It is important to observe that $Z_{24}$ cannot be obtained from $Z_{60}$ by removing points on disjoint lines. Thus, even though it is a subset of $24$ points in $Z_{60}$, it is not obtainable by the procedure described in Corollary \ref{cor: Z60 minus lines}. Indeed, the procedure diminishes in each the set $Z_{60}$ by sets of six collinear points. Since initially there are $10$ lines, each containing $6$ points from $Z_{60}$, in each step of the procedure there is certain number of these lines. On the other hand in $Z_{24}$ there are at most $4$ points in a line.

Indeed, in \eqref{eq: 18 lines} there are already $6$ lines from each family $A,B,\ldots,F$ involved.

\begin{proposition}\label{prop: Z24 is geproci}
   The set $Z_{24}$ has the \geproci \, property.
\end{proposition}
\proof
   This has been already proved in the Appendix to \cite{ChiantiniMigliore19}. However, now our understanding of the situation is better. The set $Z_{24}$ has the $\calc(4)$ property as proved in the Appendix, so its image under a general projection is contained in an irreducible curve $\Gamma$ of degree $4$. The points on $\Gamma$ can be either cut out by images of $6$ disjoint lines, as discussed above, or by an irreducible curve determined by the cone given in Theorem \ref{thm: unexpected cone of degree 6}.
\endproof
   We conclude by the following, amusing, observation.
\begin{remark}
   The residual set $Z_{60}\setminus Z_{24}$ is a $(6,6)$-grid.
\end{remark}
\proof
   It can be checked directly that the following two sets of lines
   $$\left\{\ell_4,\ell_9,\ell_{14},\ell_{17},\ell_{22},\ell_{27}\right\}\;\;\mbox{ and }\;\;
   \left\{\ell_5, \ell_8, \ell_{13}, \ell_{18}, \ell_{23}, \ell_{26}\right\}$$
   intersect in the $36$ points of $Z_{60}\setminus Z_{24}$.
\endproof

We conclude by answering an interesting question raised by the referee.
\begin{proposition}\label{prop:notR}
   The Klein configuration $Z_{60}$ can not be realized over the reals.
\end{proposition}
By the realizability we mean simply finding a set of $60$ points in $\P^3(\R)$ with the same colinearities and coplanarities as those in $Z_{60}$.
\begin{proof}
   The idea is to look at one of the $60$ planes incident to $15$ points in $Z_{60}$. To work with specific points, we choose the plane $w=0$. Then we obtain the configuration indicated in Figure \ref{fig:F}. For the sake of clarity we just indicate the numbers of the points, which are in the accordance with those introduced earlier. This configuration is an augmented Fermat configuration, see \cite[Examples 3.4 and 3.5]{Szp19c}.

\definecolor{uuuuuu}{rgb}{0.26,0.26,0.26}
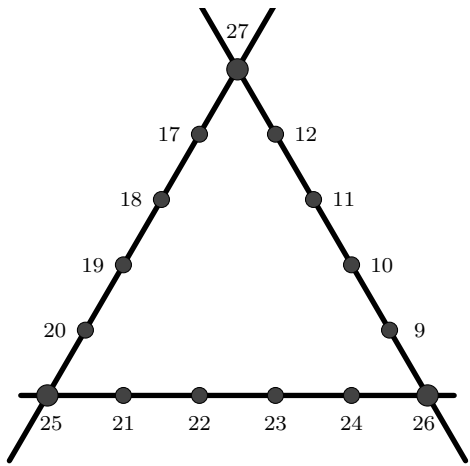
\begin{figure}[h!]
  \centering
\begin{tikzpicture}[line cap=round,line join=round,>=triangle 45,x=1cm,y=1cm]
\clip(-1,-.5) rectangle (7,6);
\draw [line width=2pt,domain=0:14.18] plot(\x,{(-0--5.196152422706633*\x)/3});
\draw [line width=2pt,domain=-4.94:6] plot(\x,{(--31.1769145362398-5.196152422706633*\x)/3});
\draw [line width=2pt,domain=.15:5.85] plot(\x,0.8660254037844393);
\begin{scriptsize}
\draw [fill=uuuuuu] (0.5,0.8660254037844393) circle (4pt);
\draw [fill=uuuuuu] (1,1.7320508075688785) circle (3pt);
\draw (0.6,1.7320508075688785) node {$20$};
\draw [fill=uuuuuu] (1.5,2.5980762113533182) circle (3pt);
\draw (1.1,2.5980762113533182) node {$19$};
\draw [fill=uuuuuu] (2,3.4641016151377584) circle (3pt);
\draw (1.6,3.4641016151377584) node {$18$};
\draw [fill=uuuuuu] (2.5,4.330127018922198) circle (3pt);
\draw (2.1,4.330127018922198) node {$17$};
\draw [fill=uuuuuu] (3.5,4.330127018922196) circle (3pt);
\draw (3.9,4.330127018922196) node {$12$};
\draw [fill=uuuuuu] (4,3.464101615137759) circle (3pt);
\draw (4.4,3.464101615137759) node {$11$};
\draw [fill=uuuuuu] (4.5,2.5980762113533244) circle (3pt);
\draw (4.9,2.5980762113533244) node {$10$};
\draw [fill=uuuuuu] (5,1.7320508075688847) circle (3pt);
\draw (5.4,1.7320508075688847) node {$9$};
\draw [fill=uuuuuu] (5.5,0.8660254037844328) circle (4pt);
\draw [fill=uuuuuu] (1.5,0.8660254037844328) circle (3pt);
\draw [fill=uuuuuu] (2.5,0.8660254037844328) circle (3pt);
\draw [fill=uuuuuu] (3.5,0.8660254037844328) circle (3pt);
\draw [fill=uuuuuu] (4.5,0.8660254037844328) circle (3pt);
\draw [fill=uuuuuu] (3,5.19) circle (4pt);
\draw (3,5.7) node {$27$};
\draw (0.55,0.5) node {$25$};
\draw (1.5,0.5) node {$21$};
\draw (2.5,0.5) node {$22$};
\draw (3.5,0.5) node {$23$};
\draw (4.5,0.5) node {$24$};
\draw (5.45,0.5) node {$26$};
\end{scriptsize}
\end{tikzpicture}
  \caption{Points from $Z_{60}$ in the $w=0$ plane}
  \label{fig:F}
\end{figure}

   Most of collinearities between points in the set
   $$F=\left\{P_9, P_{10}, P_{11}, P_{12}, P_{17}, P_{18} , P_{19}, P_{20}, P_{21}, P_{22}, P_{23}, P_{24}\right\}$$
   are not visible in Figure \ref{fig:F}. This is so for a good reason, it is impossible to draw them in the real plane and this is our argument for the Proposition.

   Turning to the details, it is easy to check by direct computations that the following triples of points are collinear
$$9,17,23;\;\; 9,18,24;\;\; 9,19,21;\;\; 9,20,22;\;\; 10,17,22;\;\; 10,18,23;\;\; 10,19,24;\;\; 10,20,21;\;\;$$
$$11,17,21;\;\; 11,18,22;\;\;
11,19,23;\;\; 11,20,24;\;\;
12,17,24;\;\; 12,18,21;\;\;
12,19,22;\;\; 12,20,23.$$
   Additionally, the quadruples indicated in Figure \ref{fig:F} are collinear
$$9,10,11,12;\;\; 17,18,19,20;\;\; 21,22,23,24.$$
Hence any line determined by a pair of points in $F$ contains at least one additional point of $F$. This is not possible over the reals due to the dual Sylvester-Gallai Theorem (see \cite[Theorem 1.1]{GreenTao2013}) and we are done.
\end{proof}

\paragraph*{Acknowledgement.}
   Our work started in the framework of the Research in Pairs by the Mathematisches Forschungsinstitut Oberwolfach in September 2020. We thank the MFO for providing excellent working conditions despite of the COVID-19 crisis.

   A preliminary version of our results was presented during Oberwolfach workshop on \emph{Lefschetz Properties in Algebra, Geometry and Combinatorics}. We thank the participants and especially Juan Migliore for very helpful comments expressed during the workshop.

   Research of P. Pokora was partially supported by National Science Centre, Poland, Sonata Grant 2018/31/D/ST1/00177. Research of T. Szemberg was partially supported by National Science Centre, Poland, Opus Grant 2019/35/B/ST1/00723.
   Research of J. Szpond was partially supported by National Science Centre, Poland, Harmonia Grant 2018/30/M/ST1/00148.

   We thank the referee for helpful remarks and suggestions which helped us to better present some parts of the manuscript.







\bigskip
\bigskip
\small

\bigskip
\noindent
   Piotr Pokora,\\
   Department of Mathematics, Pedagogical University of Cracow,
   Podchor\c a\.zych 2,
   PL-30-084 Krak\'ow, Poland.

\nopagebreak
\noindent
   \textit{E-mail address:} \texttt{piotrpkr@gmail.com}\\

\bigskip
\noindent
   Tomasz Szemberg,\\
   Department of Mathematics, Pedagogical University of Cracow,
   Podchor\c a\.zych 2,
   PL-30-084 Krak\'ow, Poland.

\nopagebreak
\noindent
   \textit{E-mail address:} \texttt{tomasz.szemberg@gmail.com}\\

\bigskip
\noindent
   Justyna Szpond,\\
   Institute of Mathematics,
   Polish Academy of Sciences,
   \'Sniadeckich 8,
   PL-00-656 Warszawa, Poland

\nopagebreak
\noindent
   \textit{E-mail address:} \texttt{szpond@gmail.com}\\


\end{document}